%% file: fb-rate.tex
\documentclass[11pt,a4paper,reqno]{article}
\usepackage{mysty}

\usepackage{fullpage}

\makeatletter
\let\reftagform@=\tagform@
\def\tagform@#1{\maketag@@@{(\ignorespaces\textcolor{black}{#1}\unskip\@@italiccorr)}}
\newcommand{\iref}[1]{\textup{\reftagform@{\tcr{\ref{#1}}}}}
\makeatother

\begin{document}

\title{Local Linear Convergence of Forward--Backward \\ under Partial Smoothness\thanks{This work has been partly supported by the European Research Council (ERC project SIGMA-Vision). JF was partly supported by Institut Universitaire de France.}}
\author{
	Jingwei Liang\thanks{Jingwei Liang, Jalal Fadili \\
	\hspace*{1.8em}GREYC, CNRS-ENSICAEN-Universit\'{e} de Caen, 
	E-mail: \{Jingwei.Liang, Jalal.Fadili\}@greyc.ensiacen.fr}
	, Jalal Fadili\samethanks 
	, Gabriel Peyr\'{e}\thanks{Gabriel Peyr\'{e} \\
	\hspace*{1.8em}CNRS, CEREMADE, Universit\'{e} Paris-Dauphine, 
	E-mail: Gabriel.Peyre@ceremade.dauphine.fr}
}
\date{}
\maketitle

\begin{abstract}
In this paper, we consider the Forward--Backward proximal splitting algorithm to minimize the sum of two proper convex functions, one of which having a Lipschitz continuous gradient and the other being partly smooth relative to an active manifold $\Mm$. We propose a generic framework under which we show that the Forward--Backward (i) correctly identifies the active manifold $\Mm$ in a finite number of iterations, and then (ii) enters a local linear convergence regime that we characterize precisely. This gives a grounded and unified explanation to the typical behaviour that has been observed numerically for many problems encompassed in our framework, including the Lasso, the group Lasso, the fused Lasso and the nuclear norm regularization to name a few. These results may have numerous applications including in signal/image processing processing, sparse recovery and machine learning.
\end{abstract}

%%%%%%%%%%%%%%%%%%%%%%%%%%%%%%%%%%%%%
\begingroup
\let\clearpage\relax

\input{tex/sec-intro-general}

\input{tex/sec-partial-smooth}

\input{tex/sec-rate-general}
\input{tex/sec-experiment}

\input{tex/sec-proof-general}

\input{tex/sec-uniqueness}

\endgroup

%\newpage
%%%%%%%%%%%%%%%%%%%%%%%%%%%%%%%%%%%%%%
\small
\bibliographystyle{plain}
\bibliography{fb-rate}

\end{document}

%% file: tex/sec-intro-general.tex
\section{Introduction}\label{eq:intro}

\subsection{Problem statement}

Convex optimization has become ubiquitous in most quantitative disciplines of science. A common trend in modern science is the increase in size of datasets,
which drives the need for more efficient optimization methods. Our goal is the generic minimization of composite functions of the form 
\beq
\label{eq:minP}\tag{$\calP$}
\min_{x\in\bbR^n} \Ba{\Phi(x) = F(x) + J(x)}  ,
\eeq
where
\begin{enumerate}[leftmargin=1.75cm, label=(\textbf{A.\arabic{*}}), ref = {\textbf{A.\arabic{*}}}]
\item \label{A-J} $J \in \lsc$, the set of proper, lower semi-continuous and convex functions;
\item \label{A-F} $F$ is a convex and $C^{1,1}(\bbR^n)$ function whose gradient is $\beta$--Lipschitz continuous;
\item \label{A-nonempty} $\Argmin \Phi \neq \emptyset$.
\end{enumerate}

The class of problem \eqref{eq:minP} covers many popular non-smooth convex optimization problems encountered in various fields throughout science and engineering, including signal/image processing, machine learning and classification. For instance, taking $F=\tfrac{1}{2\lambda}\norm{y - A \cdot}^2$ for some operator $A: \RR^{n} \to \RR^{m}$ and $\lambda > 0$, we recover the Lasso problem when $J=\norm{\cdot}_1$, the group Lasso for $J=\norm{\cdot}_{1,2}$, the fused Lasso for $J=\norm{D^*\cdot}_1$ with $D=[D_\mathrm{DIF}, \, \epsilon\Id]$ and $D_\mathrm{DIF}$ is the finite difference operator, anti-sparsity regularization when $J=\norm{\cdot}_\infty$, and nuclear norm regularization when $J=\norm{\cdot}_*$.

The standard (non-relaxed) version of the Forward--Backward (FB) splitting algorithm \cite{lions1979splitting} to solve \eqref{eq:minP} updates to a new iterate $\xkp$ based on the following rule,
\begin{equation}
\label{eq:fbs}
\xkp = \prox_{\gamma_k J} \Pa{\xk - \gamma_k \nabla F(\xk)},
\end{equation}
starting from any point $x^0 \in \bbR^n$, where $0 < \underline{\gamma} \leq \gamma_k \leq \overline{\gamma} < 2/\beta$. The proximity operator is defined as, for $\gamma>0$
\[
\prox_{\gamma J}(x) = \argmin_{z\in\bbR^n} \frac{1}{2\gamma}\norm{z-x}^2 + J(z) .
\]

%%%%%%%%%%%%%%%%%%%%%%%%%%%%%%%%%%%%%%%%%%%%%%%%%%%%%%%%
\subsection{Contributions}

In this paper, we present a unified local linear convergence analysis for the FB algorithm to solve \eqref{eq:minP} when $J$ is in addition partly smooth relative to a manifold $\Mm$ (see Definition~\ref{dfn-partly-smooth}). The class of partly smooth functions is very large and encompasses all previously discussed examples as special cases. More precisely, we first show that FB has a finite identification property, meaning that after a finite number of iterations, say $K$, all iterates obey $x_k \in \Mm$ for $k \geq K$. Exploiting this property, we then show that after such a large enough number of iterations, $x_k$ converges locally linearly. We characterize this regime and the rates precisely depending on the structure of the active manifold $\Mm$. In general, $x_k$ converges locally $Q$-linearly, and when $\Mm$ is an linear subspace, the convergence becomes $R$-linear. Several experimental results on some of the problems discussed above are provided to support our theoretical findings.

%%%%%%%%%%%%%%%%%%%%%%%%%%%%%%%%%%%%%%%%%%%%%%%%%%%%%%%%
\subsection{Related work}

Finite support identification and local $R$-linear convergence of FB to solve the Lasso problem, though in infinite-dimensional setting, is established in \cite{bredies2008linear} under either a very restrictive injectivity assumption, or a non-degeneracy assumption which is a specialization of ours (see \eqref{eq:conditionnondeg}) to the $\ell_1$-norm. A similar result is proved in \cite{hale07}, for $F$ being a smooth convex and locally $C^2$ function and $J$ the $\ell_1$-norm, under restricted injectivity and non-degeneracy assumptions. The $\ell_1$-norm is a partly smooth function and hence covered by our results. \cite{Wainwright12} proved $Q$-linear convergence of FB to solve \eqref{eq:minP} for $F$ satisfying restricted smoothness and strong convexity assumptions, and $J$ being a so-called convex decomposable regularizer. Again, the latter is a small subclass of partly smooth functions, and their result is then covered by ours. For example, our framework covers the total variation (TV) semi-norm and $\ell_{\infty}$-norm regularizers which are not decomposable.

In \cite{hare2004identifying,Hare-Lewis-Algo}, the authors have shown finite identification of active manifolds associated to partly smooth functions for various algorithms, including the (sub)gradient projection method, Newton-like methods, the proximal point algorithm. Their work extends that of e.g. \cite{Wright-IdentSurf} on identifiable surfaces from the convex case to a general non-smooth setting. Using these results, \cite{HareFB11} considered the algorithm \cite{Tseng09} to solve \eqref{eq:minP} where $J$ is partly smooth, but not necessarily convex and $F$ is $C^2(\bbR^n)$, and proved finite identification of the active manifold. However, the convergence rate remains an open problem in all these works.

%%%%%%%%%%%%%%%%%%%%%%%%%%%%%%%%%%%%%%%%%%%%%%%%%%%%%%%%
\subsection{Notations}

For a nonempty convex set $\calC \subset \RR^n$, $\ri(\Cc)$ denotes its relative interior, $\Aff(\calC)$ is its affine hull, $\LinHull(\calC)$ is the subspace parallel to it. We denote $\proj_{\calC}$ the orthogonal projector onto $\calC$, and for a matrix $A \in \RR^{m \times n}$, $A_\calC = A \circ \proj_\calC$.

Suppose $\Mm \subset \bbR^n$ is a $C^2$-manifold around $x \in \bbR^n$, denote $\tgtManif{x}{\Mm}$ the tangent space of $\Mm$ at $x \in \bbR^n$. The tangent model subspace is defined as
\beqn
T_{x} = \LinHull\Pa{\partial J(x)}^\perp .
\eeqn
It is easy to see that $\proj_{T_x}\Pa{\partial J(x)}$ is single-valued, we therefore define the generalized sign vector
\[
e_x = \proj_{T_x}\Pa{\partial J(x)}.
\]
It is straightforward to show that $e_x=\proj_{\Aff\pa{\partial J(x)}}(0)$.

%% file: tex/sec-partial-smooth.tex
\section{Partial smoothness}\label{sec:partly_smooth_func}

In addition to \iref{A-J}, our central assumption is that $J$ is a partly smooth function. Partial smoothness of functions is originally defined in \cite{Lewis-PartlySmooth}. Our definition hereafter specializes it to functions in $\lsc$.

\begin{definition}\label{dfn-partly-smooth}
Let $J \in \lsc$, and $x \in \bbR^n$ such that $\partial J(x) \neq \emptyset$.
$J$ is \emph{partly smooth} at $x$ relative to a set $\Mm$ containing $x$ if 
\begin{enumerate}[label={\rm (\arabic{*})}]
\item \label{PS-C2} (Smoothness) $\Mm$ is a $C^2$-manifold around $x$ and $J$ restricted to $\Mm$ is $C^2$ around $x$.
\item \label{PS-Sharp} (Sharpness) The tangent space $\tgtManif{x}{\Mm}$ is $T_{x}$.
\item \label{PS-DiffCont} (Continuity) The set--valued mapping $\partial J$ is continuous at $x$ relative to $\Mm$.
\end{enumerate}
In the following, the class of partly smooth functions at $x$ relative to $\Mm$ is denoted as $\PSF{x}{\Mm}$. When $\Mm$ is an affine manifold, then $\Mm = x + T_x$, and we denote this subclass as $\PSFA{x}{x+T_x}$. When $\Mm$ is a linear manifold, then $\Mm = T_x$, and we denote this subclass as $\PSFL{x}{T_x}$.
\end{definition}

Capitalizing on the results of \cite{Lewis-PartlySmooth}, it can be shown that under mild transversality assumptions, the set of continuous convex partly smooth functions is closed under addition and pre-composition by a linear operator. Moreover, absolutely permutation-invariant convex and partly smooth functions of the singular values of a real matrix, i.e. spectral functions, are convex and partly smooth spectral functions of the matrix \cite{Daniilidis-SpectralIdent}. 

It then follows that all the examples discussed in Section~\ref{eq:intro}, including $\ell_1$, $\ell_{1,2}$, $\ell_\infty$ norms, TV semi-norm and nuclear norm, are partly smooth. In fact, the nuclear norm is partly smooth at a matrix $x$ relative to the manifold $\Mm = \enscond{x'}{ \rank(x')=\rank(x)}$. The first three regularizers are all part of the class $\PSFL{x}{T_x}$, see Section~\ref{sec:experiment} and \cite{Vaiter14} for details. 

%%%%%%%%%%%%%%%%%%%%%%%%%%%%%%%%%%%%%%%%%%%%%%%%%%%%%%%%
We now define a subclass of partly smooth functions where the active manifold is actually a subspace and the generalized sign vector $e_x$ is locally constant.
\begin{definition}\label{dtf-polyhedral}
$J$ belongs to the class $\PSFLS{x}{T_x}$ if and only if $J \in \PSFA{x}{x+T_x}$ or $J \in \PSFL{x}{T_x}$ and $e_x$ is constant near $x$, i.e. there exists a neighbourhood $U$ of $x$ such that $\forall x' \in T_x \cap U$
\[
e_{x'} = e_x .
\]
\end{definition}
A typical family of functions that comply with this definition is that of partly polyhedral functions~\cite[Section~6.5]{Vaiter13}, which includes the $\ell_1$ and $\ell_\infty$ norms, and the TV semi-norm.

%% file: tex/sec-rate-general.tex
%%%%%%%%%%%%%%%%%%%%%%%%%%%%%%%%%%%%%%%%%%%%%%%%%%%%%%%%%%%%%%%%%%%%%%%%%%%%%%%%%%%%%%%%%%%%%%%%%%%%%%%%%%%%%%
\section{Local linear convergence of the FB method}

In this section, we state our main result on finite identification and local linear convergence of the FB method under partial smoothness. 

%%%%%%%%%%%%%%%%%%%%%%%%%%%%%%%%%%%%%%%%%%%%%%%%%%%%%%%%
\begin{theorem}[Local linear convergence]\label{thm:linear_rate}
Assume that \iref{A-J}-\iref{A-nonempty} hold. Suppose that the FB scheme is used to create a sequence $\xk$ which converges to $\xsol \in \Argmin \Phi$ such that $J \in \PSF{x^\star}{\Mm_{\xsol}}$, $F$ is $C^2$ near $\xsol$ and
\beq\label{eq:conditionnondeg}
-\nabla F(\xsol) \in \ri\Pa{\partial J(\xsol)} .
\eeq 
Then we have the following,
\begin{enumerate}[label=\rm(\arabic{*})]
\item The FB scheme \eqref{eq:fbs} has the finite identification property, i.e. there exists $K \geq 0$, such that for all $k \geq K$, $x_k \in \Mm_{\xsol}$.
\item Suppose moreover that $\exists \alpha > 0$ such that
\beq\label{eq:conditionstrongconv}
\proj_T\nabla^2 F(\xsol)\proj_T \succeq \alpha \Id ,
\eeq
where $T:=T_{\xsol}$. Then for all $k \geq K$, the following holds,
\begin{enumerate}[label={\rm(\roman{*})}, ref={\rm(\roman{*})}]
\item \label{Q-rate}
\emph{$Q$-linear convergence}: if $0 < \underline{\gamma} \leq \gamma_k \leq \bar{\gamma} < \min\!\Pa{2\alpha\beta^{-2},2\beta^{-1}}$, then given any $1 > \rho > \widetilde{\rho}$, 
\[
\norm{\xkp - \xsol} \leq \rho \norm{ \xk -  \xsol },
\]
where $\widetilde{\rho}^2 = \max\Ba{ q(\underline{\gamma}), q(\bar{\gamma}) } \in [0,1[$ and $q(\gamma) = 1 - {2\alpha}\gamma + {\beta^2}\gamma^2$. 
\item \label{R-rate}
\emph{$R$-linear convergence}: if $J \in \PSFA{\xsol}{\xsol+T}$ or $J \in \PSFL{\xsol}{T}$, then for $0 < \underline{\gamma} \leq \gamma_k \leq \bar{\gamma} < \min\!\Pa{2\alpha\nu^{-2},2\beta^{-1}}$, where $\nu \leq \beta$ is the Lipschitz constant of $\proj_T \nabla F \proj_T$, then
\[
\norm{\xkp-\xsol} \leq \rho_k \norm{\xk-\xsol} ,
\]
where $\rho_k^2 = {1 - {2\alpha}\gamma_k + {\nu^2}\gamma_k^2} \in [0,1[$. Moreover, if $\frac{ \alpha }{ \nu^2 } \leq \bar{\gamma}$ and set $\gamma_k \equiv \frac{ \alpha }{ \nu^2 }$, then the optimal linear rate can be achieved
\[
\rho^* = \sqrt{1 - \frac{\alpha^2}{\nu^2}} .
\]
\end{enumerate}
\end{enumerate}
\end{theorem}

%%%%%%%%%
\begin{remark}
$\;$\\\vspace*{-0.5cm}
\begin{itemize}
\item The non-degeneracy assumption in \eqref{eq:conditionnondeg} can be viewed as a geometric generalization of the strict complementarity of non-linear programming. Building on the arguments of \cite{Hare-Lewis-Algo}, it turns out that it is almost a necessary condition for finite identification of $\Mm$.
\item Under the non-degeneracy and local strong convexity assumptions \eqref{eq:conditionnondeg}-\eqref{eq:conditionstrongconv}, one can actually show that  $\xsol$ is unique, see Theorem~\ref{thm:uniq}.
\item For $F=G \circ A$, where $G$ satisfies \iref{A-F}, assumption \eqref{eq:conditionstrongconv} and the constant $\alpha$ can be restated in terms of local strong convexity of $G$ and restricted injectivity of $A$ on $T$, i.e. $\ker(A) \cap T = \ens{0}$.
\item Partial smoothness guarantees that $x_k$ arrives the active manifold in finite time, hence raising the hope of acceleration using second-order information. For instance, one can think of turning to geometric methods along the manifold $\Mm_{\xsol}$, where faster convergence rates can be achieved. This is also the motivation behind the work of e.g. \cite{miller2005newton}. 
\end{itemize}
\end{remark}

%%%%%%%%%%%%%%%%%%%%%%%%%%%%%%%%%%%%%%%%%%%%%%%%%%%%%%%%
When $J \in \PSFLS{\xsol}{T}$, it turns out that the restricted convexity assumption \eqref{eq:conditionstrongconv} of Theorem~\ref{thm:linear_rate} can be removed in some cases, but at the price of less sharp rates.

\begin{theorem}\label{thm:linear_ratepsfls}
Suppose that assumptions \iref{A-J}-\iref{A-nonempty} hold. For $\xsol \in \Argmin \Phi$, suppose that $J \in \PSFLS{\xsol}{T_{\xsol}}$, \eqref{eq:conditionnondeg} is fulfilled, and there exists a subspace $V$ such that $\ker\!\Pa{\proj_T\nabla^2 F(x)\proj_T}=V$ for any $x \in \Ball_\epsilon(\xsol)$, $\epsilon > 0$. Let the FB scheme be used to create a sequence $x_k$ that converges to $\xsol$ with $0 < \underline{\gamma} \leq \gamma_k \leq \bar{\gamma} < \min\!\Pa{2\alpha\beta^{-2},2\beta^{-1}}$, where $\alpha > 0$ (see the proof). Then there exists a constant $C > 0$ and $\rho \in [0,1[$ such that for all $k$ large enough 
\[
\norm{x_k-\xsol} \leq C \rho^k .
\]
\end{theorem}
A typical example where this result applies is when $F=G \circ A$ with $G$ locally strongly convex, in which case $V=\ker(A_T)$.  \\

We finally consider a special case of $F$ where it is a quadratic function of the form, 
\beq\label{eq:F-quadratic}
F(x) = \frac{1}{2} \norm{A x - y}^2 ,
\eeq
where $A: \RR^{n} \to \RR^{m}$ is a bounded linear operator. For this case, the rates in Theorem~\ref{thm:linear_rate} can be refined further since the gradient operator $\nabla F$ becomes linear. Let $\sigma_{\max}$ be the largest eigenvalue of $A^*A$, and $\sigma_{m},~\sigma_{M}$ be the smallest and largest eigenvalues of $A_{T}^*A_{T}$.

%%%%%%%
\begin{corollary}\label{thm:quadratic-F-fb}
Let $F$ as in \eqref{eq:F-quadratic}.
Suppose that assumptions \iref{A-J} and \iref{A-nonempty} hold. Let the FB scheme be used to create a sequence $\xk$ that converges to $\xsol \in \Argmin \Phi$ such that $J \in \PSF{x^\star}{\calM_{\xsol}}$, \eqref{eq:conditionnondeg} is fulfilled, and 
\beq\label{eq:kerA-cap-T-0}
\ker(A) \cap T_{\xsol} = \ba{0}  .
\eeq
Then there exists $K > 0$ such that for all $k \geq K$,
\begin{enumerate}[label={\rm (\arabic{*})}, ref={\rm(\arabic{*})}]
\item
\emph{$Q$-linear rate}: if $0 < \underline{\gamma} \leq \gamma_k \leq \bar{\gamma} < 2\sigma_{m}/\sigma_{\max}^{2}$, then given any $1 > \rho > \widetilde{\rho}$,
\[
\norm{\xkp - \xsol} \leq \rho\, \norm{ \xk -  \xsol } ,
\]
where $\widetilde{\rho}^2 = \max\Ba{ q(\underline{\gamma}), q(\bar{\gamma}) } \in [0,1[$ and $q(\gamma) = 1 - {2\sigma_{m}}\gamma + {\sigma_{\max}^2}\gamma^2$. 
\item
\emph{$R$-linear rate}: if $J \in \PSFA{\xsol}{\xsol+T}$ or $J \in \PSFL{\xsol}{T}$, then for $0 < \underline{\gamma} \leq \gamma_k \leq \bar{\gamma} < 2/\sigma_{\max}$,
\[
\norm{\xk-\xsol} \leq {\rho_k}\, \norm{\xk-\xsol} ,
\]
where $\rho_k = \max\!\Ba{ \abs{1-{\gamma_k}\sigma_{m}} , \abs{1-{\gamma_k}\sigma_{M}} } \in [0,1[$. 
Moreover if $\frac{2}{\sigma_{m}+\sigma_{M}} \leq \bar{\gamma}$ and we choose $\gamma_k \equiv \frac{2}{\sigma_{m}+\sigma_{M}}$, then the optimal rate can be achieved
\[
\rho^* = \frac{\varphi - 1}{\varphi + 1} = 1 - \frac{2}{\varphi+1}  ,
\]
where $\varphi = \sigma_{M}/\sigma_{m}$ is the condition number of $A_{T}^*A_{T}$.
\end{enumerate}
\end{corollary}

%% file: tex/sec-experiment.tex
\section{Numerical experiments}\label{sec:experiment}

In this section, we describe some examples to demonstrate the applicability of our results. More precisely, we consider solving 
\begin{equation} \label{eq:xlasso}
\min_{x \in \RR^n} \frac{1}{2} \norm{y - A x}^2 + \lambda J(x)
\end{equation}
where $y \in \bbR^m$ is the observation, $A: \RR^{n} \to \RR^{m}$, $\lambda$ is the tradeoff parameter, and $J$ is either the $\ell_1$-norm, the $\ell_{\infty}$-norm, the $\ell_1-\ell_2$-norm, the TV semi-norm or the nuclear norm.

%%%%%%%%%%%%%%%%%%%%%%%%%%%%%%%%%%%%%%%%%%%%%%%%%%%%%%%%
\begin{example}[$\ell_1$-norm]
For $x\in\bbR^n$, the sparsity promoting $\ell_1$-norm \cite{chen1999atomi,tibshirani1996regre} is
\beqn
  J(x) = \sum_{i=1}^n |x_i| .
\eeqn
It can verified that $J$ is a polyhedral norm, and thus $J \in \PSFLS{x}{T_x}$ for the model subspace 
\beqn
  \Mm = T_x =\enscond{u \in \RR^n}{\supp(u) \subseteq \supp(x)}, ~~\mathrm{and}~~ e_x = \sign\pa{x}. 
\eeqn
The proximity operator of the $\ell_1$-norm is given by a simple soft-thresholding.
\end{example}

%%%%%%%%%%%%%%%%%%%%%%%%%%%%%%%%%%%%%%%%%%%%%%%%%%%%%%%%
\begin{example}[$\ell_1-\ell_2$-norm]
The $\ell_1-\ell_2$-norm is usually used to promote group-structured sparsity \cite{yuan2005model}. Let the support of $x\in\bbR^n$ be divided into non-overlapping blocks $\Bb$ such that $\bigcup_{b \in \Bb} b = \ba{1,\ldots,n}$. The $\ell_1-\ell_2$-norm is given by
\beqn
  J(x) = \norm{x}_\Bb = \sum_{b \in \Bb} \norm{x_b} ,
\eeqn
where $x_b = (x_i)_{i \in b} \in \RR^{|b|}$. $\norm{\cdot}_\Bb$ in general is not polyhedral, yet partly smooth relative to the linear manifold
\beqn
\Mm = T_x = \enscond{u \in \RR^n}{\supp_\Bb(u) \subseteq \supp_\Bb(x)}, ~~\mathrm{and}~~ e_x = \Pa{\calN\pa{x_b}}_{b\in\Bb} ,
\eeqn
where $\supp_\Bb(x) = \bigcup \enscond{b}{ x_{b} \neq 0 }$, $\calN\pa{x} = x/\norm{x}$ and $\calN\pa{0} = 0$. The proximity operator of the $\ell_1-\ell_2$ norm is given by a simple block soft-thresholding.
% More details of group sparsity can be found in \cite{bach2008consistency}. 
\end{example}

%%%%%%%%%%%%%%%%%%%%%%%%%%%%%%%%%%%%%%%%%%%%%%%%%%%%%%%%
\begin{example}[Total Variation]
As stated in the introduction, partial smoothness is preserved under pre-composition by a linear operator. Let $J_0$ be a closed convex function and $D$ is a linear operator. Popular examples are the TV semi-norm in which case $J_0=\norm{\cdot}_1$ and $D^*=D_\mathrm{DIF}$ is a finite difference approximation of the derivative \cite{rudin1992nonlinear}, or the fused Lasso for $D=[D_\mathrm{DIF} ,\, \epsilon\Id]$ \cite{tibshirani2004sparsity}.

If $J_0 \in \PSF{D^* x}{\Mm_0}$, then it is shown in~\cite[Theorem 4.2]{Lewis-PartlySmooth} that under an appropriate transversality condition, $J \in \PSF{x}{\Mm}$ where
\beqn
  \Mm = \enscond{u \in \RR^n}{D^* u \in \Mm_0 } .
\eeqn
In particular, for the case of the TV semi-norm, we have $J \in \PSFLS{x}{T_x}$ with
\beqn
  \Mm = T_x = \enscond{u \in \RR^n}{\supp(D^* u) \subseteq I} ~~\mathrm{and}~~ e_x = \proj_{T_x}D\sign\pa{D^*x}
\eeqn
where $I = \supp\pa{D^*x}$. The proximity operator for the 1D TV, though not available in closed form, can be obtained efficiently using either the taut string algorithm \cite{DaviesKovac01} or the graph cuts \cite{Chambolle12}.
\end{example}

%%%%%%%%%%%%%%%%%%%%%%%%%%%%%%%%%%%%%%%%%%%%%%%%%%%%%%%%
\begin{example}[$\ell_{\infty}$-norm]
For $x\in\bbR^n$, the anti-sparsity promoting $\ell_{\infty}$-norm is defined as following
\[
J(x) = \max_{1 \leq i \leq N} \abs{x_i}  .
\]
It plays a prominent role in a variety of applications including approximate nearest neighbor search \cite{jegou2012anti} or vector quantization \cite{lyubarskii2010uncertainty}, see also \cite{studer2012signal} and references therein.

It can verified that $J$ is a polyhedral norm, and thus $J \in \PSFLS{x}{T_x}$ for the model subspace
\beqn
\Mm = T_x =\enscond{\alpha}{ \alpha_{(I)} = r s_{(I)} ,~  r \in \RR }, ~~\mathrm{and}~~ e_x = \frac{s}{\abs{I}},
\eeqn
where $s = \sign(x)$ and $I = \enscond{i}{ \abs{x_i} = \norm{x}_\infty }$.
The proximity operator of the $\ell_{\infty}$-norm is given by the difference between itself and the projection onto $\ell_1$-ball.
\end{example}

%%%%%%%%%%%%%%%%%%%%%%%%%%%%%%%%%%%%%%%%%%%%%%%%%%%%%%%%
\begin{example}[Nuclear norm]
Low-rank is the spectral extension of vector sparsity to matrix-valued data $x \in \RR^{n_1 \times n_2}$, i.e. imposing the sparsity on the singular values of $x$. Let $x=U \Lambda_x V^*$ a reduced singular value decomposition (SVD) of $x$. The nuclear norm of a $x$ is defined as
\eq{
  J(x) = \norm{x}_* = \sum_{i=1}^{r} (\Lambda_x)_i , 
}
where $\rank\pa{x} = r$. It has been used for instance as SDP convex relaxation for many problems including in machine learning \cite{bach2008consistency,grave2011trace}, matrix completion \cite{recht2010guaranteed,candesExactCompletion} and phase retrieval \cite{CandesPhaseLift}.

It can be shown that the nuclear norm is partly smooth relative to the manifold \cite[Example~2]{LewisMalick08},
\beqn
\Mm = \enscond{z \in \bbR^{n_1 \times n_2}}{\rank(z) = r } .
\eeqn
The tangent space to $\calM$ at $x$ and $e_x$ are given by
\[
\tgtManif{x}{\Mm} = \enscond{z \in \bbR^{n_1\times n_2}}{z=U L^* + M V^*, ~ \forall L \in \RR^{n_2 \times r}, M \in \RR^{n_1 \times r}} , ~~\mathrm{and}~~ e_x = UV^* .
\]
The proximity operator of the nuclear norm is just soft--thresholding applied to the singular values.
\end{example}

%%%%%%%%%%%%%%%%%%%%%%%%%%%%%%%%%%%%%%%%%%%%%%%%%%%%%%%%
\paragraph*{Recovery from random measurements}
In these examples, the forward observation model is
\beq \label{eq:forward-model}
y = A x_0 + \varepsilon, \quad \varepsilon \sim \calN(0,\delta^2) ,
\eeq
where $A \in \bbR^{m\times n}$ is generated uniformly at random from the Gaussian ensemble with i.i.d. zero-mean and unit variance entries. The tested experimental settings are
\begin{enumerate}[label = (\alph{*})]
\item \textbf{$\ell_1$-norm} $m=48$ and $n=128$, $x_0$ is $8$-sparse;
\item \textbf{Total Variation} $m=48$ and $n=128$, $(D_\mathrm{DIF} x_0)$ is $8$-sparse;
\item \textbf{$\ell_\infty$-norm} $m=123$ and $n=128$, $x_0$ has $10$ saturating entries;
\item \textbf{$\ell_1-\ell_2$-norm} $m=48$ and $n=128$, $x_0$ has $2$ non-zero blocks of size $4$;
\item \textbf{Nuclear norm} $m=1425$ and $n=2500$, $x_0 \in \bbR^{50\times50}$ and $\rank(x_0) = 5$.
\end{enumerate}
The number of measurements is chosen sufficiently large, $\delta$ small enough and $\lambda$ of the order of $\delta$ so that \cite[Theorem~1]{Vaiter14} applies, yielding that the minimizer of \eqref{eq:xlasso} is unique and verifies the non-degeneracy and restricted strong convexity assumptions \eqref{eq:conditionnondeg}-\eqref{eq:conditionstrongconv}. 

The convergence profile of $\norm{\xk - \xsol}$ are depicted in Figure~\ref{fig:fb}(a)-(e). Only local curves after activity identification are shown. For $\ell_1$, TV and $\ell_\infty$, the predicted rate coincides exactly with the observed one. This is because these regularizers are all partly polyhedral gauges, and the data fidelity is quadratic, hence making the predictions of Theorem~\ref{thm:linear_rate}(ii) exact. For the $\ell_1-\ell_2$-norm, although its active manifold is still a subspace, the generalized sign vector $e_k$ is not locally constant, which entails that the the predicted rate of Theorem~\ref{thm:linear_rate}(ii) slightly overestimates the observed one. For the nuclear norm, whose active manifold is not linear, thus Theorem~\ref{thm:linear_rate}(i) applies, and the observed and predicted rates are again close.

\paragraph*{TV deconvolution} 
In this image processing example, $y$ is a degraded image generated according to the same forward model as \eqref{eq:xlasso}, but now $A$ is a convolution with a Gaussian kernel. The anisotropic TV regularizer is used. The convergence profile is shown in Figure~\ref{fig:fb}(f). Assumptions \eqref{eq:conditionnondeg}-\eqref{eq:conditionstrongconv} are checked a posteriori. This together with the fact that the anisotropic TV is polyhedral  justifies that the predicted rate is again exact.

\begin{figure}[!htb]
\centering
\subfloat[$\ell_1$ (Lasso)]{\includegraphics[width=0.33\linewidth]{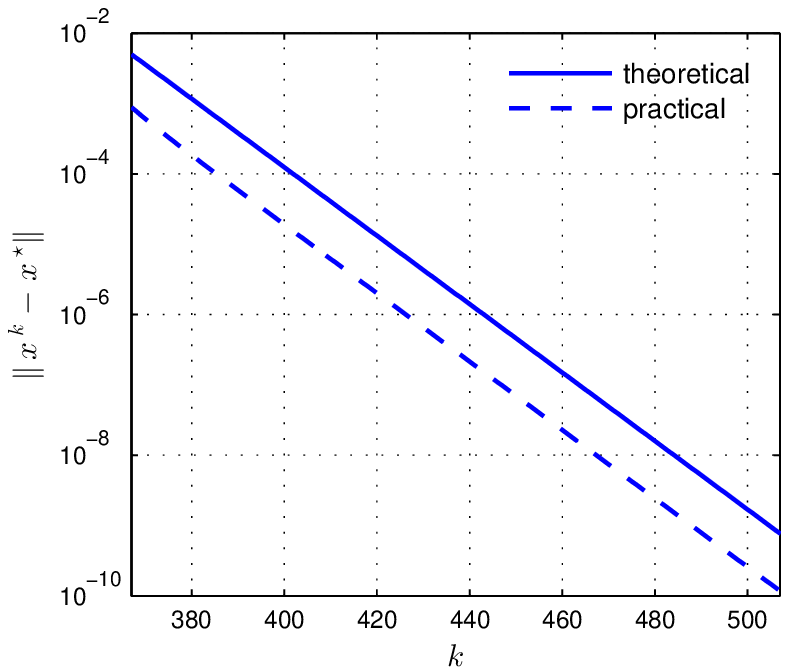}} 
\subfloat[TV semi-norm]{\includegraphics[width=0.33\linewidth]{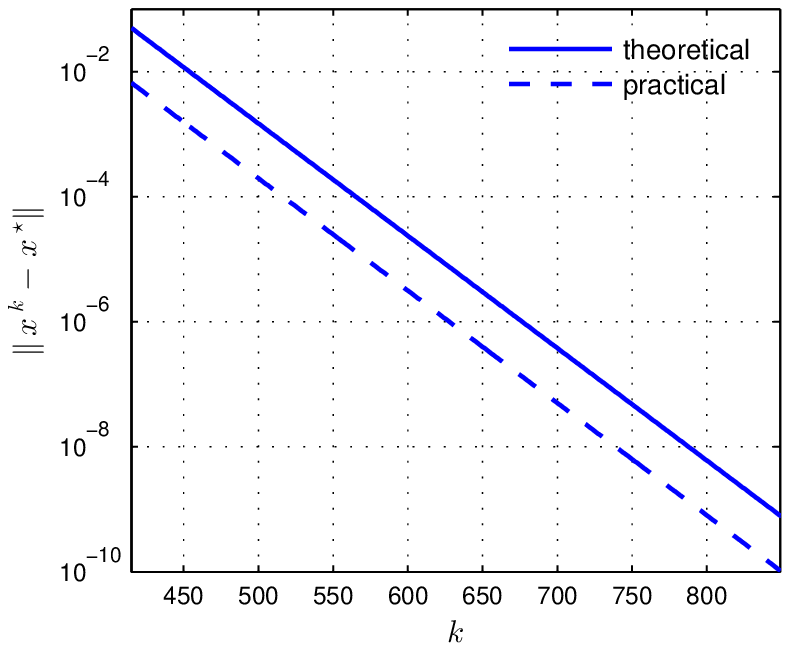}} 
\subfloat[$\ell_{\infty}$-norm]{\includegraphics[width=0.33\linewidth]{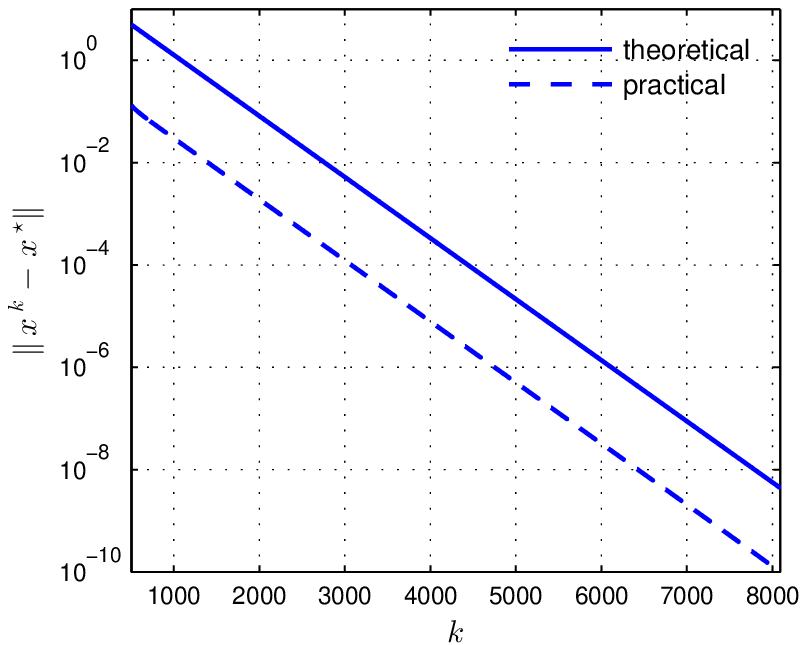}}  \\
%%%%%%%%
\subfloat[$\ell_1-\ell_2$-norm]{\includegraphics[width=0.33\linewidth]{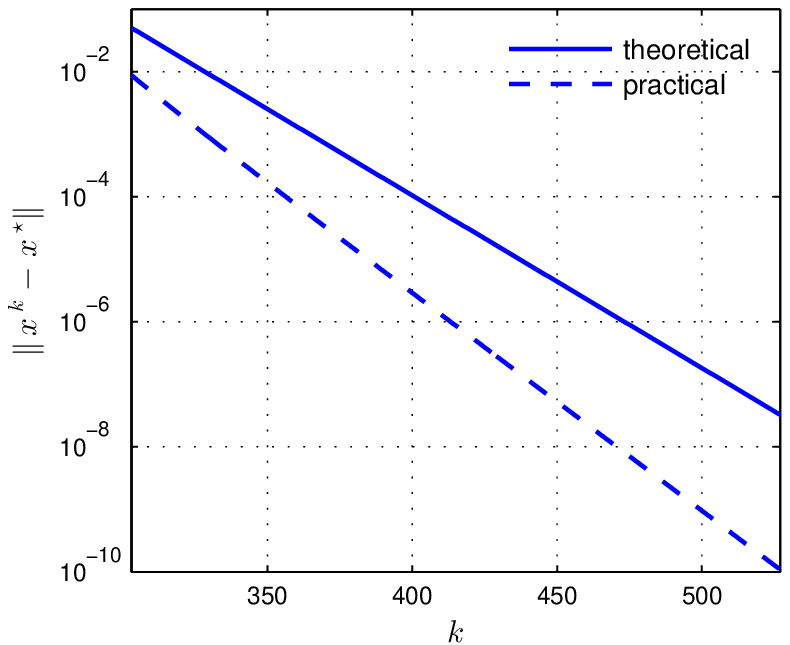}}
\subfloat[Nuclear norm]{\includegraphics[width=0.33\linewidth]{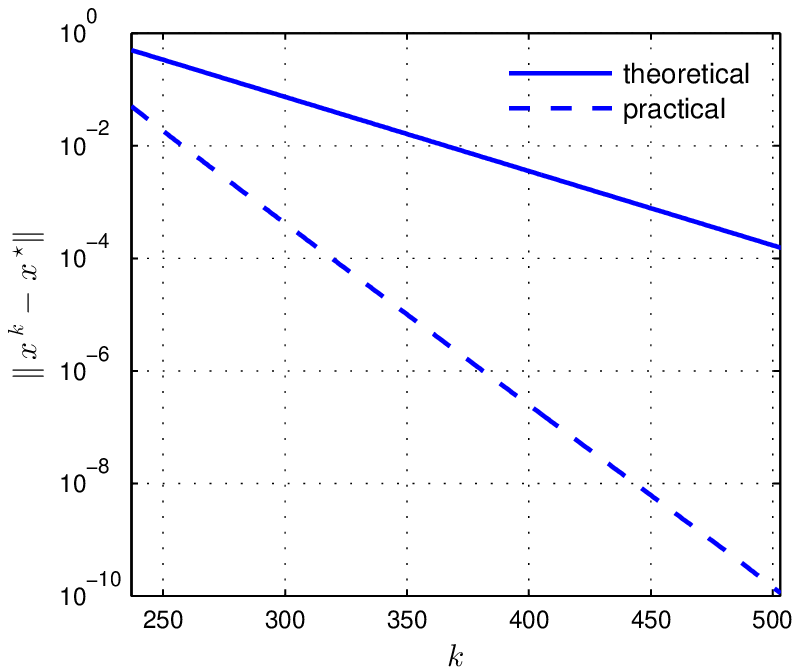}}
\subfloat[TV deconvolution]{\includegraphics[width=0.33\linewidth]{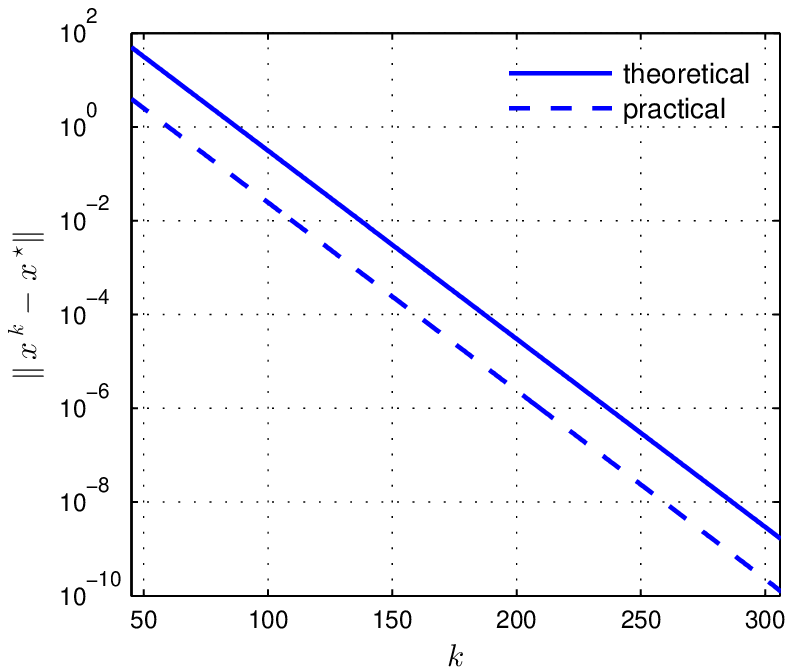}}
\caption{Observed and predicted local convergence profiles of the FB method \eqref{eq:fbs} in terms of $\norm{\xk - \xsol}$ for different types of partly smooth functions. (a) $\ell_1$-norm; (b) TV semi-norm; (c) $\ell_{\infty}$-norm; (d) $\ell_{1}-\ell_2$-norm; (e) Nuclear norm; (f) TV deconvolution.}
\label{fig:fb}
\end{figure}

%% file: tex/sec-proof-general.tex
\section{Proofs of the main results}\label{sec:proof}

We start with the following useful lemma.

%%%%%%%%%%%%%%%%%%%%%%%%%%%%%%%%%%%%%%%%%%%%%%%%%%%%%%%%%%%%%%%%%%%%%%%%%%%%%%%%%%%%%%%%%%%%%%%%%%%%%%%%%%%%
\begin{lemma}\label{lemma:proj_M}
Suppose that $J \in \PSF{x}{\Mm}$. Then for any $x' \in \Mm \cap U$, where $U$ is a neighborhood of $x$, the projector $\proj_{\calM}(x')$ is uniquely valued and $C^1$ around $x$, and thus 
\[
x' - x = \proj_{T_x}(x'-x) + o\Pa{\norm{x'-x}} .
\]
If $J \in \PSFL{x}{T_x}$, then
\[
x' - x = \proj_{T_x}(x'-x) .
\]
\end{lemma}
%%%%%%%%
\begin{proof}%[Proof of Lemma \ref{thm:proj_M}]
Partial smoothness implies that $\Mm$ is a $C^2$--manifold around $x$, then $\proj_{\calM}(x')$ is uniquely valued \cite{PoliquinRockafellar2000} and moreover $C^1$ near $x$ \cite[Lemma 4]{LewisMalick08}. Thus, continuous differentiability shows
\eq{
x' - x = \proj_{\Mm}(x') - \proj_{\Mm}(x) = \deriv \proj_{\Mm}(x)(x - x') + o\Pa{\norm{x - x'}} ,
} 
where $\deriv \proj_{\Mm}(x)$ is the derivative of $\proj_{\Mm}$ at $x$. By virtue of \cite[Lemma 4]{LewisMalick08} and the sharpness property of $J$, this derivative is given by
\eq{
\mathrm{D} \proj_{\Mm}(x) = \proj_{\tgtManif{x}{\Mm}} = \proj_{T_x} .
}
The case where $\calM$ is affine or linear is immediate. This concludes the proof.
\end{proof}

%%%%%%%%%%%%%%%%%%%%%%%%%%%%%%%%%%%%%%%%%%%%%%%%%%%%%%%%%%%%%%%%%%%%%%%%%%%%%%%%%%%%%%%%%%%%%%%%%%%%%%%%%%%%
\begin{proof}[Proof of Theorem~\ref{thm:linear_rate}]
$\;$\\\vspace*{-0.5cm}
\begin{enumerate}[leftmargin=0.8cm,label={\rm (\arabic{*})}, ref={\rm(\arabic{*})}]
\item 
Classical convergence results of the FB scheme, e.g. \cite{combettes2005signal}, show that $x_k$ converges to some $\xsol \in \Argmin \Phi \neq \emptyset$ by assumption \iref{A-nonempty}. Assumptions \iref{A-J}-\iref{A-F} entail that \eqref{eq:conditionnondeg} is equivalent to $0 \in \ri \Pa{\partial\Phi(\xsol)}$. Since $F \in C^{2}$ around $\xsol$, the smooth perturbation rule of partly smooth functions \cite[Corollary~4.7]{Lewis-PartlySmooth} ensures that $\Phi \in \PSF{\xsol}{\Mm}$. By definition of $x_{k+1}$, we have
\beqn
\frac{1}{\gamma_k} \Pa{G_k(x_k) - G_k(x_{k+1})} \in \partial \Phi (x_{k+1}) .
\eeqn
where $G_k = \Pa{\Id - \gamma_k \nabla F}$. By Baillon-Haddad theorem, is an $\alpha$-averaged operator, hence non-expansive, which yields
\begin{align*}
\mathrm{dist}\Pa{0,\partial \Phi(x_{k+1})}
&\leq \frac{1}{\gamma_k} \norm{G_k(x_k) - G_k(x_{k+1})} 
\leq \frac{1}{\gamma_k} \norm{x_k - x_{k+1}} .
\end{align*}
Since $\liminf \gamma_k = \underline{\gamma} > 0$, we obtain $\mathrm{dist}\Pa{0,\partial \Phi(x_{k+1})} \to 0$. Owing to assumptions \iref{A-J}-\iref{A-F}, $\Phi$ is sub-differentially continuous at every point in its domain \cite[Example~13.30]{rockafellar1998variational}, and in particular at $\xsol$ for $0$, and thus $\Phi(x_k) \to \Phi(\xsol)$. Altogether, this shows that the conditions of \cite[Theorem 5.3]{hare2004identifying} are fulfilled, whence the claim follows.
\item 
Since $\prox_{\gamma_k J}$ is firmly non-expansive, hence non-expansive, we have
\begin{align}
\norm{\xkp - \xsol}
= \norm{\prox_{\gamma_k J}G_k \xk - \prox_{\gamma_k J}G_k \xsol}
\leq \norm{G_k \xk - G_k \xsol} . \label{eq:ineq}
\end{align}
\begin{enumerate}[leftmargin=0.6cm,label={\rm (\roman{*})}, ref={\rm(\roman{*})}]
\item By virtue of Lemma~\ref{lemma:proj_M}, we have
\[
\xk - \xsol = \proj_{T}\pa{\xk - \xsol} + o\Pa{\norm{\xk-\xsol}} .
\]
This, together with local $C^2$ smoothness of $F$ and Lipschitz continuity of $\nabla F$ entails
\begin{align}
&\iprod{\xk - \xsol}{\nabla F(\xk)  -  \nabla F(\xsol)}	\nonumber \\
=~& \int_{0}^1 \iprod{\xk - \xsol}{\nabla^2 F(\xsol + t (\xk - \xsol)) (\xk - \xsol)} dt \nonumber \\
=~& \int_{0}^1 \langle \proj_T (\xk - \xsol) + o\Pa{\norm{\xk-\xsol}} , \nonumber\\&\qquad\quad \nabla^2 F(\xsol + t (\xk - \xsol))\proj_T (\xk - \xsol) + o\Pa{\norm{\xk-\xsol}} \rangle dt \nonumber \\
=~& \int_{0}^1 \iprod{\proj_T (\xk - \xsol)}{\nabla^2 F(\xsol + t (\xk - \xsol))\proj_T (\xk - \xsol)} dt + o\Pa{\norm{\xk-\xsol}^2} \nonumber \\
\geq~& \alpha \norm{\xk - \xsol}^2 + o\Pa{\norm{\xk-\xsol}^2} . \label{eq:dprodM}
\end{align}
Since \eqref{eq:conditionstrongconv} holds and $\nabla^2 F(x)$ depends continuously on $x$, there exists $\epsilon > 0$ such that $\proj_T\nabla^2 F(x)\proj_T \succeq \alpha \Id$, $\forall x \in \Ball_{\epsilon}(\xsol)$. Thus, classical development of the right hand side of \eqref{eq:ineq} yields
\begin{align}
& \norm{\xkp - \xsol}^2 \leq \norm{G_k \xk - G_k \xsol}^2
= \norm{ \pa{\xk - \xsol} - {\gamma_k} \pa{\nabla F(\xk)  -  \nabla F(\xsol) } }^2 \nonumber \\
=~& \norm{\xk - \xsol}^2 - {2\gamma_k}\iprod{\xk - \xsol}{\nabla F(\xk)  -  \nabla F(\xsol) } + {\gamma_k^2} \norm{\nabla F(\xk)  -  \nabla F(\xsol) }^2 \nonumber \\
\leq~& \norm{\xk - \xsol}^2 - {2\gamma_k\alpha}\norm{\xk - \xsol}^2 + {\gamma_k^2\beta^2} \norm{\xk - \xsol }^2 + o\Pa{\norm{\xk-\xsol}^2} \nonumber \\
=~& \Pa{ 1 - {2\alpha}\gamma_k + {\beta^2}\gamma_k^2 } \norm{ \xk -  \xsol }^2 + o\Pa{\norm{\xk-\xsol}^2} . \label{eq:Gkx_norm}
\end{align}
Taking the lim sup in this inequality gives
\begin{align}\label{eq:qlinearlimsup}
\limsup_{k\rarrow\pinf}\frac{\norm{\xkp - \xsol}^2 }{\norm{ \xk -  \xsol }^2} \leq q(\gamma_k) = 1 - {2\alpha}\gamma_k + {\beta^2}\gamma_k^2 .
\end{align}
It is clear that for $0 < \underline{\gamma} \leq \gamma \leq \bar{\gamma} < \min\Pa{2\alpha\beta^{-2},2\beta^{-1}}$, $q(\gamma) \in [0, 1[$, and $q(\gamma) \leq \widetilde{\rho}^2=\max\Ba{ q(\underline{\gamma}), q(\bar{\gamma}) }$. Inserting this in \eqref{eq:qlinearlimsup}, and using classical arguments yields the result.
%
%
%
%%%%%%%%%%%%%%%%%%%%%%%%%%%%%
\item 
We give the proof for $\Mm=T$, that for $\Mm=\xsol+T$ is similar. Since $\xk$ and $\xsol$ belong to $T$, from $\xkp = \prox_{\gamma_k J}\pa{G_k \xk}$ we have
\beqn
G_k \xk - \xkp \in \gamma_k \partial J(\xkp) 
\Longrightarrow  \xkp = \proj_{T} \Pa{ G_k \xk - \gamma_k \partial J(\xkp) } = \proj_{T} G_k \xk - \gamma_k e_{k+1} .
\eeqn
Similarly, we have $\xsol =  \proj_{T} G_k \xsol - \gamma_k e^\star$. We then arrive at
\begin{align}\label{eq:xk_xsol}
\pa{\xkp - \xsol} + \gamma_k\pa{e_{k+1}-e^\star} 
&= \proj_T \pa{G_k \xk - G_k \xsol}  \\
&= \pa{\xk - \xsol} - {\gamma_k} \Pa{ \proj_T\nabla F (\proj_T \xk) - \proj_T\nabla F(\proj_T \xsol) } . \nonumber
\end{align}
Moreover, maximal monotonicity of $\gamma_k\partial J$ gives
\[
\begin{aligned}
& \norm{\pa{\xkp-\xsol} + \gamma_k\pa{e_{k+1}-e^\star} }^2 \\
=~& \norm{ \xkp-\xsol}^2 + 2\iprod{\xkp-\xsol}{\gamma_k\pa{e_{k+1}-e^\star}} + \gamma_k\norm{e_{k+1}-e^\star}^2
\geq \norm{ \xkp-\xsol }^2 .
\end{aligned}
\]
It is straightforward to see that now, \eqref{eq:dprodM} becomes
\[
\iprod{\xk - \xsol}{\nabla F(\proj_T \xk) - \nabla F(\proj_T \xsol)} \geq \alpha \norm{\xk - \xsol}^2 .
\]
Let $\nu$ be the Lipschitz constant of $\proj_T \nabla F \proj_T$, and obviously $\nu \leq \beta$.
Developing $\norm{\proj_T \pa{G_k \xk - G_k \xsol}}^2$ similarly to \eqref{eq:Gkx_norm} we obtain
\[
\norm{\xkp - \xsol}^2 \leq \Pa{1 - {2\alpha}\gamma_k + {\nu^2}\gamma_k^2} \norm{ \xk -  \xsol }^2 = \rho_k^2 \norm{ \xk -  \xsol }^2 ,
\]
where $\rho_k \in [0,1[$ for $0 < \underline{\gamma} \leq \gamma_k \leq \bar{\gamma} < \min\Pa{2\alpha\nu^{-2},2\beta^{-1}}$. 
$\rho_k$ is minimized at $\frac{\alpha}{\nu^2}$ whenever it obeys the given upper-bound, whence the optimal rate follows after rearranging the terms,
\[
\rho^* 
= \sqrt{ q\Ppa{\frac{\alpha}{\nu^2}} }
= \sqrt{ 1 - \frac{\alpha^2 }{\nu^2} }  . \qedhere
\]
\end{enumerate}
\end{enumerate}
\end{proof}

%%%%%%%%%%%%%%%%%%%%%%%%%%%%%%%%%%%%%%%%%%%%%%%%%%%%%%%%%%%%%%%%%%%%%%%%%%%%%%%%%%%%%%%%%%%%%%%%%%%%%%%%%%%%
\begin{proof}[Proof of Theorem~\ref{thm:linear_ratepsfls}]
Arguing similarly to the proof of Theorem~\ref{thm:linear_rate}(ii), and using in addition that $e^\star=e^{\xsol}$ is locally constant, we get
\begin{align*}
\xkp - \xsol	&= \pa{\xk - \xsol} - {\gamma_k} \Pa{ \proj_T \nabla F(\proj_T \xk) - \proj_T\nabla F(\proj_T \xsol) }  \\
		&= \pa{\xk - \xsol} - {\gamma_k} \int_{0}^1 \proj_T \nabla^2 F(\xsol + t (\xk - \xsol)) \proj_T (\xk - \xsol) dt
\end{align*}
Denote $H_t=\proj_T \nabla^2 F(\xsol + t (\xk - \xsol)) \proj_T \succeq 0$. Using that $H_t$ is self-adjoint, we have
\[
\proj_{V} \xkp = \proj_{V}\xk .
\]
Since $\xk \to \xsol$, it follows that $\proj_{V}\xk = \proj_{V}\xsol$ for all $k$ sufficiently large. Observing that $\xk - \xsol = \proj_{V^\perp}(\xk - \xsol)$ for all large $k$, we get
\[
\xkp - \xsol = \xk - \xsol - {\gamma_k} \int_{0}^1 \proj_{V^\perp} H_t \proj_{V^\perp} (\xk - \xsol) dt .
\]
Observe that $V^\perp \subset T$. By definition, $B_t=H_t^{1/2} \proj_{V^\perp}$ is injective, and therefore, $\exists \alpha > 0$ such that $\norm{B_t x}^2 > \alpha \norm{x}^2$ for all $x \neq 0$ and $t \in [0,1]$. We then have
\begin{align*}
&\norm{\xkp - \xsol}^2 \\
=~& \norm{\xk - \xsol}^2 - {2\gamma_k}\int_{0}^1\iprod{\xk - \xsol}{B_t^TB_t (\xk - \xsol)} dt + {\gamma_k^2} \norm{ \proj_{V^\perp}\proj_T \Pa{\nabla F(\proj_T \xk) - \nabla F(\proj_T \xsol)}}^2  \\
=~& \norm{\xk - \xsol}^2 - {2\gamma_k}\int_{0}^1\norm{B_t(\xk - \xsol)}^2 dt + {\gamma_k^2} \norm{\proj_{V^\perp}\proj_T}^2 \norm{\nabla F(\xk)  -  \nabla F( \xsol)}^2  \\
=~& \norm{\xk - \xsol}^2 - {2\gamma_k\alpha}\norm{\xk - \xsol}^2 + {\gamma_k^2\norm{\proj_T\proj_{V^\perp}}^2}\norm{\nabla F(\xk)  -  \nabla F( \xsol)}^2  \\
\leq~& \norm{\xk - \xsol}^2 - {2\gamma_k\alpha}\norm{\xk - \xsol}^2 + {\gamma_k^2\beta^2\norm{\proj_{V^\perp}}^2} \norm{\proj_{V^\perp}(\xk - \xsol) }^2  \\
\leq~& \norm{\xk - \xsol}^2 - {2\gamma_k\alpha}\norm{\xk - \xsol}^2 + {\gamma_k^2\beta^2}\norm{ \xk - \xsol}^2  \\
=~& \Pa{ 1 - {2\alpha}\gamma_k + {\beta^2}\gamma_k^2 } \norm{ \xk -  \xsol }^2 = \rho_k^2 \norm{ \xk -  \xsol }^2 . 
\end{align*}
It is easy to see again that $\rho_k \in [0,1[$ whenever $0 < \underline{\gamma} \leq \gamma_k \leq \bar{\gamma} < \min\Pa{2\beta^{-1},2\alpha\beta^{-2}}$.
\end{proof}

%%%%%%%%%%%%%%%%%%%%%%%%%%%%%%%%%%%%%%%%%%%%%%%%%%%%%%%%%%%%%%%%%%%%%%%%%%%%%%%%%%%%%%%%%%%%%%%%%%%%%%%%%%%%
\begin{proof}[Proof of Corollary \ref{thm:quadratic-F-fb}]
$\;$\\\vspace*{-0.5cm}
\begin{enumerate}[label={\rm (\arabic{*})}, ref={\rm(\arabic{*})}]
\item 
We have $\nabla F(x) = A^*(Ax - y)$. Applying Theorem~\ref{thm:linear_rate} with $\beta=\sigma_{\max}$ and $\alpha=\sigma_{m}$, where $\alpha > 0$ owing to \eqref{eq:kerA-cap-T-0}, and using the fact that $\sigma_{m}/\sigma_{\max} \leq \sigma_{m}/\sigma_{M} \leq 1$, we get the desired claim.

\item 
We give the proof for $\Mm=T$, that for $\Mm=\xsol+T$ is similar. From \eqref{eq:xk_xsol} we have
\beqn
\pa{\xkp-\xsol} + \gamma_k\pa{e_{k+1}-e^\star} = \Pa{\Id - \gamma_k A_T^*A_T} \pa{\xk - \xsol}  ,
\eeqn
which leads to
\[
\norm{\xkp - \xsol} \leq \norm{ \Id - {\gamma_k}A^*_TA_T }_2 \norm{ \xk - \xsol }  ,
\]
Since $\Id - {\gamma_k}A^*_TA_T$ is symmetric, then
\[
\norm{ \Id - {\gamma_k}A^*_TA_T }_2 = \max\Ba{ \abs{1-{\gamma_k}\sigma_{m}} , \abs{1-{\gamma_k}\sigma_{M}} }  .
\]
Consider the following piecewise linear function of $\gamma$,
\[
\ell(\gamma) = 
\max\Ba{ \abs{1-{\gamma}\sigma_{m}} , \abs{1-{\gamma}\sigma_{M}} }
\]
we have using that $0 < \underline{\gamma} \leq \gamma \leq \bar{\gamma} < \frac{2}{\sigma_{\max}}$ and \eqref{eq:kerA-cap-T-0}, $\ell(\gamma_k) \in [0,1[$.
Therefore let
\[
\rho = \max\Ba{ \ell(\underline{\gamma}) , \ell(\bar{\gamma}) } ,
\]
we have
\beqn
\norm{\xkp - \xsol} \leq \rho \norm{\xk - \xsol} \leq \rho^{k+1} \norm{x_{0} - \xsol}  . 
\eeqn
Moreover if $\frac{2}{\sigma_{m}+\sigma_{M}} \leq \bar{\gamma}$, then the best rate can be achieved is
\[
\rho^* 
= \frac{ \sigma_{M} - \sigma_{m} }{\sigma_{M} + \sigma_{m}} 
= \frac{\varphi - 1}{\varphi + 1}  
= 1 - \frac{2}{\varphi + 1} . \qedhere
\]
\end{enumerate}
\end{proof}

%% file: tex/sec-uniqueness.tex
\appendix
\section{Uniqueness}\label{sec:uniq}

\begin{theorem} \label{thm:uniq}
Let $\xsol$ a minimizer of \eqref{eq:minP}, where $J$ is a proper closed convex function and $F \in C^{1}(\bbR^n)$ is convex and $C^2$ around $\xsol$. Let $T:=T_{\xsol}$. Suppose that the non-degeneracy \eqref{eq:conditionnondeg} and local strong convexity assumption \eqref{eq:conditionstrongconv} hold. Then $\xsol$ is the unique solution of \eqref{eq:minP}.
\end{theorem}

%%%%%%%%%%%%%%%%%%%%%
\begin{proof}
Since $F$ is locally $C^2$ around $\xsol$, there exists $\epsilon > 0$ sufficiently small such that for any $\delta \in \Ball_\epsilon(0)$, we have
\begin{align*}
\Phi(\xsol+\delta) - \Phi(\xsol) 	
&= F(\xsol+\delta) - F(\xsol) - \iprod{\nabla F(\xsol)}{\delta} + J(\xsol+\delta) - J(\xsol) + \iprod{\nabla F(\xsol)}{\delta} \\
&= \frac{1}{2} \iprod{\delta}{\nabla^2 F(\xsol + t\delta)\delta} + J(\xsol+\delta) - J(\xsol) + \iprod{\nabla F(\xsol)}{\delta}, \quad t \in ]0,1[ ~.
\end{align*}
Let $x_t=\xsol+ t\delta \in \Ball_\epsilon(\xsol)$. Since \eqref{eq:conditionstrongconv} holds and $\nabla^2 F(x)$ depends continuously on $x \in \Ball_\epsilon(\xsol)$, we have $\proj_T\nabla^2 F(x)\proj_T \succeq \alpha \Id$ for any such $x$. This holds in particular at $x_t$. We then distinguish two cases.
\begin{itemize}%[label=(\arabic{*})]
\item $\delta \notin \ker(\nabla^2 F(x_t))$. In this case, it is clear that
\[
\Phi(\xsol+\delta) - \Phi(\xsol) \geq \frac{1}{2} \iprod{\delta}{\nabla^2 F(x_t)\delta}  \geq \alpha/2 \norm{\delta}^2 > 0
\]
since $F$ is convex and locally $C^2$, and $J$ is convex with $-\nabla F(\xsol) \in \partial J(\xsol)$.
\item $\delta \in \ker(\nabla^2 F(x_t)) \setminus \ens{0}$. Since $J$ is a proper closed convex function, it is sub-differentially regular at $\xsol$. Moreover $\partial J(\xsol) \neq \emptyset$ ($-\nabla F(\xsol)$ is in it), and thus the directional derivative $J'(\xsol,\cdot)$ is proper and closed, and it is the support of $\partial J(\xsol)$ \cite[Theorem~8.30]{rockafellar1998variational}. It then follows from the separation theorem \cite[Theorem V.2.2.3]{hiriart1996convex} that
\begin{equation*}
-\nabla F(\xsol) \in \ri(\partial J(\xsol)) \iff J'(\xsol,\delta) > -\iprod{\nabla F(\xsol)}{\delta} \quad \forall \delta \text{ s.t. } J'(\xsol;\delta) + J'(\xsol;-\delta) > 0 .    
\end{equation*}
As $\ker(J'(\xsol;\cdot)) = \T$ \cite[Proposition~3(iii) and Lemma~10]{Vaiter13}, and in view of \eqref{eq:conditionstrongconv}, we get
\begin{align*}
-\nabla F(\xsol) \in \ri(\partial J(\xsol)) 	
&\iff J'(\xsol;\delta) > -\iprod{\nabla F(\xsol)}{\delta} \quad \forall \delta \notin T \\
&\Longrightarrow J'(\xsol;\delta) > -\iprod{\nabla F(\xsol)}{\delta} \quad \forall \delta \in \ker(\nabla^2 F(x_t)) \setminus \ens{0}.
\end{align*}
Combining this with classical properties of the directional derivative of a convex function yields
\begin{align*}
\Phi(\xsol+\delta) - \Phi(\xsol) 
&= J(\xsol+\delta) - J(\xsol) + \iprod{\nabla F(\xsol)}{\delta} \\
&\geq J'(\xsol;\delta) + \iprod{\nabla F(\xsol)}{\delta} > 0 ,
\end{align*}
which concludes the proof. \qedhere
\end{itemize}
\end{proof}